\documentclass[12pt]{article}
\usepackage[textwidth=16cm, textheight=22cm]{geometry}
\usepackage{amssymb}
\usepackage{amsmath}
\usepackage{amsthm}
\usepackage{hyperref}
\usepackage{algorithm}
\usepackage{algorithmic}
\usepackage[dvips]{graphicx}

\newcommand{\RR}{\mathbb{R}}
\newcommand{\nnegm}{\in\RR^{n\times n}_+}
\newcommand{\nnegv}{\in\RR^{n}_+}
\newcommand{\hmu}{\hat{\mu}}
\newcommand{\setC}{\mathcal C}
\newcommand{\setD}{\mathcal D}
\newcommand{\Psim}{\Psi_{\otimes}^p}
\newcommand{\hPsim}{\hat\Psi_{\otimes}^p}
\newcommand{\diag}{\operatorname{diag}}
\newtheorem{theorem}{Theorem}[section]
\newtheorem{lemma}{Lemma}[section]
\newtheorem{proposition}{Proposition}[section]
\newtheorem{corollary}{Corollary}[section]
\newtheorem{example}{Example}[section]
\newtheorem{remark}{Remark}[section]

\newtheorem{definition}{Definition}[section]

\begin{document}

\title{The Analytic Hierarchy Process, Max Algebra and Multi-objective Optimisation}

\author{Buket Benek Gursoy\thanks{Hamilton Institute, National University of Ireland,
Maynooth, Co. Kildare, Ireland. Email: buket.benek@nuim.ie}\and%
Oliver Mason\thanks{Corresponding author. Hamilton Institute, National University of Ireland,
Maynooth, Co. Kildare, Ireland. Email: oliver.mason@nuim.ie}\and%
Serge\u{\i} Sergeev\thanks{University of Birmingham, School of Mathematics, Edgbaston B15 2TT.
Email: sergiej@gmail.com}}

\date{}
\maketitle



\begin{abstract}
The Analytic Hierarchy Process (AHP) is widely used for decision making involving multiple criteria.  Elsner and van den Driessche~\cite{Drie3, Drie4} introduced a max-algebraic approach to the single criterion AHP.  We extend this to the multi-criteria AHP, by considering multi-objective generalisations of the single objective optimisation problem solved in these earlier papers.  We relate the existence of globally optimal solutions to the commutativity properties of the associated matrices; we relate min-max optimal solutions to the generalised spectral radius; and we prove that Pareto optimal solutions are guaranteed to exist. \\



\end{abstract}

\section{Introduction}
\label{sec:intro}

The analytic hierarchy process (AHP) is a method for ranking alternatives in multi-criteria decision making problems.  Developed by Saaty \cite{saaty77}, it consists of a three layer hierarchical structure: the overall goal is at the top; the criteria are in the next level; and the alternatives are in the bottom level.  The AHP has been used in  many different areas including manufacturing systems, finance, politics, education, business and industry; for more details on the method, see the monographs by Saaty-Vargas and Vaidya-Kumar~\cite{saaty01, VaidKum}.

The essence of the AHP can be described as follows.  Given $n$ alternatives we construct a {\it{pairwise comparison matrix}} ({\it{PC}}-matrix), $A>0$ for each criterion, in which $a_{ij}$ indicates the strength of alternative $i$ relative to alternative $j$ for that criterion.  A {\it{PC}}-matrix with the property that $a_{ij}a_{ji}=1$ for all $i\neq j$ and $a_{ii}=1$ for all $i$ is called a {\it{symmetrically reciprocal}} matrix ({\it{SR}}-matrix) \cite{Farkas}. (Note that this abbreviation might clash with the strongly regular matrices 
of Butkovi\v{c}~\cite{ButkovicBook}, but not in this paper.) 

Once an {\it{SR}}-matrix $A$ is constructed, the next step in the AHP is to derive a vector $(w_1, \ldots , w_n)$ of positive weights, which can be used to rank the alternatives, with $w_i$ quantifying the weight of alternative $i$.  As observed by
Elsner and van~den~Driessche~\cite{Drie3}, the ideal situation is where $a_{ij} = w_i/w_j$, in which case the {\it{SR}}-matrix is {\it transitive}.  In practice, this will rarely be the case and it is necessary to approximate $A$ with a transitive matrix $T$, where $t_{ij} = w_i/w_j$ for some positive weight vector $w=(w_1, \ldots , w_n)$.  The problem is then how to construct $T$ given $A$.  Several approaches have been proposed including Saaty's suggestion to take $w$ to be the Perron vector of $A$, or the approach of Farkas et al.~\cite{Farkas}, which chooses $w$ to minimise the Euclidean error $\sum\limits_{i,j} (a_{ij} - w_i/w_j)^2$. Elsner and 
van~den~Driessche~\cite{Drie3, Drie4} suggested selecting $w$ to be the max algebraic eigenvector of $A$.  This is similar in spirit to Saaty's approach and also generates a transitive matrix that minimises the maximal relative error $\max\limits_{i,j}|a_{ij} - w_i/w_j|/a_{ij}$.  As noted in \cite{Drie4}, 
minimising this functional is equivalent to minimising
\begin{equation}
\label{eq:errf}
e_A(x)=\max\limits_{1\leq i,j\leq n}a_{ij}x_j/x_i.
\end{equation}
The different approaches to approximating an {\it {SR}}-matrix $A$ with a transitive matrix $T$ will in general produce different rankings of the alternatives.  The question of how these rankings are affected by the choice of scheme is considered in the recent paper of Ngoc~\cite{Ngoc}.  

In the classical AHP involving multiple criteria, a set of {\it{SR}}-matrices is constructed: one for each criterion.  One additional {\it{SR}}-matrix is constructed based on comparisons of the different criteria.  Once weight vectors are obtained for each individual criterion, these are then combined using the entries of the weight vector for the criteria-comparison matrix.  As an illustration, we take the following numerical example from Saaty~\cite{saaty77} and show how the Perron vectors of the comparison matrices are used to construct a weight vector.

\begin{example}
\label{ex:Saaty}
The problem considered is deciding where to go for a one week vacation among the alternatives: 1. Short trips, 2. Quebec, 3. Denver, 4. California. Five criteria are considered: 1. cost of the trip, 2. sight-seeing opportunities, 3. entertainment, 4. means of travel and 5. dining. The {\it{PC}}-matrix for the criteria and its Perron vector are given by
$$C=\left [ \begin{array}{ccccc}
   1 & 1/5 & 1/5 & 1 & 1/3 \\
   5 & 1 & 1/5 & 1/5 & 1 \\
   5 & 5 & 1 & 1/5 & 1 \\
   1 & 5 & 5 & 1 & 5 \\
   3 & 1 & 1 & 1/5 & 1 \\
\end{array} \right ]\quad \text{and}\quad
c=\left [ \begin{array}{c}
   0.179  \\ 0.239  \\ 0.431  \\ 0.818 \\  0.237  \\
\end{array} \right ].
$$
The above matrix $C$ describes the pairwise comparisons between the different \emph{criteria}.  For instance, as $c_{21} = 5$, criterion 2 is rated more important than criterion 1; $c_{32} = 5$ indicates that criterion 3 is rated more important than criterion 2 and so on.  The vector $c$ contains the weights of the criteria; in this method, criterion 4 is given most weight, followed by criterion 3 and so on.

The {\it{SR}}-matrices, $A_1, ..., A_5$,  for each of the 5 criteria, their Perron vectors and corresponding ranking schemes are given below. For instance, for criterion 1, the first alternative is preferred to the second as the $(1, 2)$ entry of $A_1$ is $3$.  Similarly, for criterion 3, the 4th alternative is preferred to the 1st as the $(4, 1)$ entry of $A_3$ is $2$.  

For the cost of the trip:
\begin{equation*}
A_1=\left [ \begin{array}{cccc}
   1 & 3 & 7 & 9 \\
   1/3 & 1 & 6 & 7 \\
   1/7 & 1/6 & 1 & 3 \\
   1/9 & 1/7 & 1/3 & 1 \\
\end{array} \right ],   \quad
v^{(1)}=\left [ \begin{array}{c}
   0.877  \\ 0.46  \\ 0.123  \\ 0.064 \\
\end{array} \right ],
  \quad 1>2>3>4
\end{equation*}

For the sight-seeing opportunities:
\begin{equation*}
A_2=\left [ \begin{array}{cccc}
   1 & 1/5 & 1/6 & 1/4 \\
   5 & 1 & 2 & 4 \\
   6 & 1/2 & 1 & 6 \\
   4 & 1/4 & 1/6 & 1 \\
\end{array} \right ],   \quad
v^{(2)}=\left [ \begin{array}{c}
   0.091  \\ 0.748  \\ 0.628  \\ 0.196 \\
\end{array} \right ],
  \quad 2>3>4>1
\end{equation*}

For the entertainment:
\begin{equation*}
A_3=\left [ \begin{array}{cccc}
   1 & 7 & 7 & 1/2 \\
   1/7 & 1 & 1 & 1/7 \\
   1/7 & 1 & 1 & 1/7 \\
   2 & 7 & 7 & 1 \\
\end{array} \right ],  \quad
v^{(3)}=\left [ \begin{array}{c}
    0.57 \\ 0.096  \\ 0.096  \\ 0.81 \\
\end{array} \right ],
  \quad 4>1>2=3
\end{equation*}

For the means of travel: 
\begin{equation*}
A_4=\left [ \begin{array}{cccc}
   1 & 4 & 1/4 & 1/3 \\
   1/4 & 1 & 1/2 & 3 \\
   4 & 2 & 1 & 3 \\
   3 & 1/3 & 1/3 & 1 \\
\end{array} \right ],  \quad
v^{(4)}=\left [ \begin{array}{c}
    0.396 \\  0.355 \\ 0.768  \\ 0.357 \\
\end{array} \right ],
  \quad 3>1>4>2
\end{equation*}

For the dining:
\begin{equation*}
A_5=\left [ \begin{array}{cccc}
   1 & 1 & 7 & 4 \\
   1 & 1 & 6 & 3 \\
   1/7 & 1/6 & 1 & 1/4 \\
   1/4 & 1/3 & 4 & 1 \\
\end{array} \right ], \quad
v^{(5)}=\left [ \begin{array}{c}
    0.723 \\ 0.642  \\ 0.088  \\ 0.242 \\
\end{array} \right ],
  \quad 1>2>4>3
\end{equation*}
To obtain the overall weight vector, we compute the weighted sum $\sum\limits_{i=1}^5 c_i v^{(i)}$.  This gives  
$$w=\left [ \begin{array}{c}
   0.919 \\ 0.745  \\ 0.862 \\ 0.757 \\
\end{array} \right ]
$$
with the associated ranking: $1 > 3 > 4 > 2$.
\end{example}
Our work here is inspired by the max-algebraic approach to the AHP introduced by Elsner and van~den~Driessche~\cite{Drie3, Drie4} and extends it in the following manner.  In~\cite{Drie3, Drie4}, the max eigenvector is used as a weight vector for a \emph{single criterion} and it is shown to be optimal in the sense of minimising the maximal relative error as discussed above.  This work naturally raises the question of how to treat multiple criteria within the max-algebraic framework.  We address this question here by considering the multi-criteria AHP as a multi-objective optimisation problem, in which we have an objective function of the form (\ref{eq:errf}) for each criterion (and associated {\it SR}-matrix).  Rather than combining individual weight vectors as in Example \ref{ex:Saaty}, we consider three approaches within the framework of multi-objective optimisation, and use the optimal solution as a weight vector in each case.  The advantage of this approach is that the weight vector can be interpreted in terms of the maximal relative error functions (\ref{eq:errf}) associated with the {\it SR}-matrices given as data for the problem.  
The optimisation problems we consider are the following.  First, we investigate the existence of a single {\it{transitive}} matrix with a minimum distance to all matrices in the set simultaneously.  We remark that this amounts to finding a common subeigenvector of the given matrices.  Clearly, this will not in general be possible.  The second problem we consider is to obtain a transitive matrix that minimises the maximal distance to any of the given {\it{SR}}-matrices.  The third problem concerns the existence of a transitive matrix that is Pareto optimal for the given set of matrices.  To illustrate our results, we revisit Example \ref{ex:Saaty} towards the end of the paper.

\section{Notation and Mathematical Background}
\label{sec:prelim}
The set  of all nonnegative real numbers is denoted by $\RR_+$; the set of all $n$-tuples of nonnegative real numbers is denoted by $\RR_+^n$ and the set of all ${n\times n}$ matrices with nonnegative real entries is denoted by $\RR_+^{n\times n}$.  We denote the set of all $n$-tuples of positive real numbers by $\textrm{int}(\RR^n_+)$.  For $A\in \RR_+^{n{\times}n}$ and $1\le i, j \le n$, $a_{ij}$ refers to the $(i, j)^{\text{th}}$ entry of $A$. The matrix $A=[a_{ij}]$ is nonnegative (positive) if $a_{ij} \ge 0$ ($a_{ij} >0$) for $1\le i, j \le n$. This is denoted by $A\nnegm$ ($A > 0$).

The weighted directed graph of $A$ is denoted by $D(A)$. It is an ordered pair $(N(A), E(A))$ where $N(A)$ is a finite set of nodes $\{1, 2, ..., n\}$ and $E(A)$ is a set of directed edges, with an edge $(i, j)$ from $i$ to $j$ if and only if $a_{ij} > 0$. A path is a sequence of distinct nodes $ i_1,i_2, \ldots, i_k$ of length $k-1$ with the weight $a_{i_1i_2}a_{i_2i_3} \cdots a_{i_{k-1}i_k}$, where $(i_p, i_{p+1})$ is an edge in $D(A)$ for $p = 1, \ldots, k-1$. It is standard that $A$ is an irreducible matrix if and only if there is a directed path between any two nodes in $D(A)$. A cycle $\Gamma$ of length $k$ is a closed path of the form $i_1, i_2, ..., i_k, i_1$. The $k^{th}$ root of its weight is called its cycle geometric mean. For a matrix $A\nnegm$, the maximal cycle geometric mean over all possible cycles in $D(A)$ is denoted by $\mu(A)$. A cycle with the maximum cycle geometric mean is called a critical cycle. Nodes that lie on some critical cycle are known as critical nodes and denoted by $N^C(A)$. The set of edges belonging to critical cycles are said to be critical edges and denoted by $E^C(A)$. The critical matrix of $A$ \cite{Drie1, Drie2}, $A^C$, is formed from the submatrix of $A$ consisting of the rows and columns corresponding to critical nodes as follows. Set $a^C_{ij} = a_{ij}$ if $(i, j)$ lies on a critical cycle and $a^C_{ij} = 0$ otherwise. We use the notation $D^C(A)$ for the critical graph where $D^C(A)=D(A^C)=(N^C(A), E^C(A))$.

The max algebra consists of the set of nonnegative numbers together with the two basic operations $a\oplus b = \max(a,b)$ and $ a\otimes b = ab $.  This is isomorphic to the max-plus algebra \cite{SandL, ButkovicBook} via the natural isomorphism $x \rightarrow \textrm{log}(x)$.  These operations extend to nonnegative matrices and vectors in the obvious manner \cite{SandL, Bapat, CGreen}. For $A$ in $\RR_+^{n\times n}$, the eigenequation in the max algebra is given by $A\otimes x = \lambda x,\quad x \geq 0, \quad\lambda \geq 0.$ $\mu(A)$ is the largest max eigenvalue of $A$ \cite{ButkovicBook}. If $A$ is irreducible, then it is the unique max eigenvalue of $A$ and there is a positive max eigenvector $x\in \textrm{int}(\RR^n_+)$ corresponding to it \cite{SandL, Bapat}. The eigenvector $v$ is unique up to a scalar multiple if and only if $A^C$ is irreducible. 

Observe that an {\it{SR}}-matrix $A$ is irreducible and $\mu(A)\ge 1$ \cite{Drie3}.
Although our primary interest is in {\it{SR}}-matrices, it is noteworthy that many of our results also hold true for non-zero reducible matrices. 

For $A\nnegm$ with $\mu(A)\le 1$, $I\oplus A\oplus A_{\otimes}^2\oplus ...$ converges to a finite matrix called the Kleene star of $A$ given by $A^*=I\oplus A\oplus A_{\otimes}^2\oplus ... \oplus A_{\otimes}^{n-1}$ where $\mu(A^*)= 1$ \cite{SandL, CGreen, SergeiVS}. Here, $a^*_{ij}$ is the maximum weight of a path from $i$ to $j$ of any length \cite{MPatWork} (if $i\neq j$), 
and $A_{\otimes}^k$ denotes the $k^{\text{th}}$ max-algebraic power of $A$.
Note that for each $A\nnegm$, if $A^*$ is finite then the max-algebraic sum of
all of the columns of $A^*$ is positive. For $A\nnegm$, the set of subeigenvectors of $A$ associated with $\mu(A)$ is called a subeigencone of $A$ and denoted by $V^*(A)=\{y\nnegv\mid A\otimes y\le \mu(A)y\}$ \cite{SergeiVS} . It was shown in Proposition 2.5 of \cite{SergeiVS} that for $A\nnegm$ with $\mu(A)=1$, $V^*(A)=V(A^*)=\operatorname{span}_{\oplus}(A^*)$ where $V(A^*)$ denotes the eigencone of $A^*$ consisting of its max eigenvectors. Note that the above-mentioned max-algebraic
sum of all columns of $A^*$ is in $V^*(A)$~\cite{SergeiVS}, so $V^*(A)$ contains positive vectors.
Note that if $\mu(A)>0$, we can normalise $A$ by $\mu(A)$ and $V^*(\frac{A}{\mu(A)})=V^*(A)$.

To the authors' knowledge, max-algebraic subeigenvectors appeared in the works of Gaubert~\cite{GaubertThes, GaubertRopt}. However,
they can be traced back to earlier works on nonnegative matrix scaling, see references in Butkovi\v{c}-Schneider~\cite{ButSchn}.

For $A\nnegm$ we will consider the following set, which was introduced in \cite{Drie4}
\begin{equation}
\label{eq:CAr}
\setC_{A,r}=\{x \in \textrm{int}(\RR^n_+)\mid A\otimes x\le r x\}.
\end{equation}
For the special case of $r = \mu(A)$, $\setC_{A,\mu(A)}$ is denoted by $\setC_A$ \cite{Drie4}. Obviously $\setC_A$
is the positive part of $V^*(A)$ (which is non-empty as we argued above), and it coincides with
$V^*(A)$ when $A$ is irreducible.  To be consistent with the notation of \cite{Drie4}, we recall the definition of the normalised set
\begin{equation}
\label{eq:DAr}
\setD_{A, r} = \{x \in \setC_{A, r} \mid x_1 = 1\}.
\end{equation}
As above, $\setD_A$ is used to denote the special case where $r = \mu(A)$.


The relations between the sets $\setC_{A, r}$, the error function (\ref{eq:errf}) and $\mu(A)$ were clarified by Elsner and van~den~Driessche~\cite{Drie4} and are recalled in the following propositions, which we easily extend (based on~\cite{ButkovicBook, ButSchn, SergeiVS})
to the general reducible case.
\begin{proposition}[cf. \cite{Drie4} Lemma 2]
\label{pro:CArprop}
Let $A\nnegm$ be nonzero. Then:
\begin{itemize}
\item[{(i)}] $\setC_{A,r}\neq \emptyset \iff r>0,$ $r\ge \mu(A)$;
\item[{(ii)}] $x\in \setC_{A,r}\iff x\in \operatorname{int}(\RR^n_+),$ $e_A(x)\le r$.
\end{itemize}
\end{proposition}
\begin{proof}
To prove(i), exploit~\cite{ButkovicBook}~Theorem~1.6.29 stating that
\begin{equation}
\label{eq:CW}
\mu(A)=\min\{\lambda\mid A\otimes x\leq \lambda x, \ x\in\operatorname{int}(\RR_+^n),\},
\end{equation}
when $\mu(A)>0$ (based on Butkovi\v{c}-Schneider~\cite{ButSchn}~Theorem 2.6).  In the trivial case $r=\mu(A)=0$, we have $V^*(A)=\{x\mid A\otimes x=0\}$, which consists of all vectors $x$ such that
$x_i\neq 0$ if and only if the $i^{\text{th}}$ column of $A$ is zero. In this case $\setC_{A,r}=\setD_{A,r}=\emptyset$, unless
$A=0$ (which we exclude).

The result of (ii) follows from the definitions of~$e_A(x)$ and~$\setC_{A,r}$ given in~\eqref{eq:errf} and~\eqref{eq:CAr}. 
\end{proof}

See also Gaubert~\cite{GaubertThes} Ch. IV Lemma 1.3.8, Krivulin~\cite{Kriv05} Lemma~1,
and Nussbaum~\cite{NusLAA} Theorem 3.1 (in a more general nonlinear context)
for closely related statements. 

\begin{proposition}[cf. \cite{Drie4}, Theorem 1 part 6]
\label{pro:unique} Let $A\nnegm$. Then $x\in \setD_A$ is unique if and only if $A^C$ is irreducible and $N^C(A)=N(A)$.
\end{proposition}
\begin{proof}
Proposition~\ref{pro:CArprop}~(i) implies that we have $\setD_A=\setC_A=\emptyset$ when $\mu(A)=0$,
so we can assume $\mu(A)>0$ and, further, $\mu(A)=1$. According to~\cite{SergeiVS}~Theorem~2.8,
\begin{equation}
\label{eq:v*gen}
V^*(A)=V(A^*)=\left\{\bigoplus_{i\in M(A)}\lambda_i g^i\oplus\bigoplus_{i\notin N^C(A)} \lambda_j g^j\mid\lambda_i,\lambda_j\in\RR_+\right\}.
\end{equation}
Here $g^i$ is the $i^{\text{th}}$ column of $A^*$.  The subset
$M(A)\subset \{1,\ldots,n\}$ is such that for each (maximal) irreducible submatrix of $A^C$ there is a unique
index of that submatrix in $M(A)$. 
By the same theorem of~\cite{SergeiVS}, based on the classical results in~\cite{SandL, CGreen},
columns of $A^*$ with indices
in the same irreducible submatrix of $A^C$ are proportional to each other, and there is no proportionality
between different $g^i$ appearing in~\eqref{eq:v*gen} (moreover, these $g^i$ are strongly linearly 
independent~\cite{ButkovicBook}).

`If'': As $A^C$ is irreducible and all nodes are critical, all columns of $A^*$ are proportional to
each other, and~\eqref{eq:v*gen} shows that $V^*(A)$ is just one ray. $A\geq A^C$ is 
irreducible as well, so $V^*(A)=\setC_A$ and $\setD_A$ is the unique vector on that ray with $x_1=1$.
   
``Only if'':  Let $A^C$ be reducible, or let $N^C(A)\neq N(A)$.  The vector $z=\bigoplus\limits_{i=1}^n g^i$
is positive. Consider $g^i\oplus\epsilon z$
for all $i$, taking small enough $\epsilon$ so that all entries of $\epsilon z$ are less than
any nonzero entry of $A^*$. In this case the positive entries of $g^i$ do not change, and if
$g^i$ and $g^j$ are not proportional then neither are $g^i\oplus\epsilon z$ and $g^j\oplus\epsilon z$.
After suitably normalising these vectors, we obtain two different vectors in
$\setD_A$. 
\end{proof}

Next, let $\Psi \subset \RR_+^{n\times n}$ be a finite set of nonnegative matrices given by
\begin{equation}
\label{eq:Psi}
\Psi=\{A_1, A_2, ..., A_m\},\quad\exists i\colon A_i\neq 0
\end{equation}

Given $\Psi$, 
let $\Psim$ denote the set of all products of matrices from $\Psi$ of length $p\ge 1$.  Formally, $\Psim= \{ A_{j_1} \otimes \cdots \otimes A_{j_p}: 1 \leq j_k \leq m\mbox{ for } 1 \leq k \leq p\}.$ Using this, the max version of the generalised spectral radius \cite{GenSpeRad1, GenSpeRad2} is defined by
\begin{equation}
\label{eq:mupsi}
\hmu(\Psi)=\limsup_{p\rightarrow\infty}(\max_{\psi\in\Psim} \mu(\psi))^{\frac{1}{p}}.
\end{equation}
Before stating the next theorem, let $S$ be the matrix given by
\begin{equation}
\label{eq:Smat}
S=\bigoplus\limits_{A \in \Psi} A = A_1\oplus A_2\oplus ...\oplus A_m.
\end{equation}

Note that $S>0$ if at least one $A_i>0$. Moreover, $\mu(S)>0$ if at least one
$\mu(A_i)>0$ and $S$ is irreducible if at least one $A_i$ is irreducible: these represent the main cases in which 
we are interested.

\begin{theorem}
\label{thm:muMSPsiS}
Let $\Psi$ be given by (\ref{eq:Psi}) and $S$ be given by (\ref{eq:Smat}).
Then, $\hmu(\Psi)=\mu(S)$. (Gaubert~\cite{GaubertPerf}, Benek~Gursoy and Mason~\cite{Paper2})
\end{theorem} 

Inspired by the approach to the single criterion AHP adopted in \cite{Drie4}, we associate a set $\setC_{\Psi, r}$ with the set $\Psi$ of nonnegative matrices and show how the geometric properties of $\setC_{A, r}$ discussed in \cite{Drie4} extend to this new setting.

Define
\begin{equation}
\label{eq:CPr}
\setC_{\Psi,r}=\{x \in \textrm{int}(\RR^n_+)\mid e_{A_i}(x)\le r {\text{ for all } A_i\in\Psi}\} = \bigcap\limits_{i=1}^{m}\setC_{A_i, r}.
\end{equation}
We also consider the set of normalised vectors:
\begin{equation}
\label{eq:DPr}
\setD_{\Psi,r}=\{x\in \setC_{\Psi, r}\mid x_1=1\}.
\end{equation}
We will use the notations $\setC_{\Psi}$ for $\setC_{\Psi,\hat\mu(\Psi)}$ and $\setD_{\Psi}$ for $\setD_{\Psi,\hat\mu(\Psi)}$.
The following result shows that the set $\setC_{\Psi, r}$ and the set $\setC_{S, r}$ are equal.  This will allow us to readily extend properties of $\setC_{A, r}$ established in \cite{Drie4} to sets of matrices.

\begin{theorem}
\label{thm:subs}
Consider the set $\Psi\subset\RR_{+}^{n\times n}$ given by (\ref{eq:Psi}),  
and let $S$ be given by (\ref{eq:Smat}).  
Then:
$$\setC_{\Psi,r}=\setC_{S,r}.$$
\end{theorem}
\begin{proof}
(i): Let $x$ in $\setC_{\Psi, r}$ be given. Then, $e_{A_i}(x)\leq r$ which implies $A_i\otimes x\leq rx$ for each $A_i\in\Psi$ from (ii) in Proposition \ref{pro:CArprop}.  Taking the maximum of both sides from $1$ to $m$, we see that $S\otimes x\leq rx$.  It follows that $x\in \setC_{S,r}$. Thus, $\setC_{\Psi,r}\subset \setC_{S,r}$.

Now choose some $x\in \setC_{S, r}$.  Then $e_{S}(x)\leq r$ from (ii) in Proposition \ref{pro:CArprop}. 
Since $e_{A_i}(x)\leq e_{S}(x)\leq r$ for all $1\leq i\leq m$, we obtain $x\in \setC_{\Psi, r}$. Thus, $\setC_{S,r}\subset \setC_{\Psi,r}$. Hence $\setC_{\Psi,r}=\setC_{S,r}$.
\end{proof}

The following corollary extends Proposition \ref{pro:CArprop} to a set of nonnegative matrices $\Psi$.  Since $\hmu(\Psi)=\mu(S)$ by Theorem~\ref{thm:muMSPsiS}~(ii), the corollary is an immediate consequence of Theorem~\ref{thm:subs}.
\begin{corollary}
\label{cor:CPrprop} Consider the set $\Psi\subset\RR_{+}^{n\times n}$ given by (\ref{eq:Psi}), 
and let $S$ be given by (\ref{eq:Smat}). Then:
\begin{itemize}
\item[{(i)}] $\setC_{\Psi,r}\neq \emptyset \iff r>0$, $r\ge \hmu(\Psi)$;
\item[{(ii)}] $x\in \setC_{\Psi,r}\iff x\in \operatorname{int}(\RR^n_+),$  $e_S(x)\le r$.
\end{itemize}
\end{corollary}

Theorem \ref{thm:subs} establishes that $\setC_{\Psi, r} = \setC_{S, r}$.  It is immediate that we also have $\setD_{\Psi,r}=\setD_{S,r}, \setC_{\Psi}=\setC_{S}$ and $\setD_{\Psi}=\setD_{S}$.  Therefore, studying these sets for a collection of 
nonnegative matrices reduces to studying the sets associated with the single matrix $S$.  This fact means that the properties of $\setC_{A, r}$ discussed in Theorem 1 of \cite{Drie4} can be directly extended to $\setC_{\Psi, r}$.  We state some of these in the following theorem, noting that property (i) does not require irreducibility. 

\begin{theorem}[cf. \cite{Drie4} Theorem 1]
\label{thm:geop}
Let $\Psi\subset\RR_+^{n\times n}$ be given in (\ref{eq:Psi}), and let $S$ be given by (\ref{eq:Smat}).  Then:
\begin{itemize}
\item[(i)] {$\setC_{\Psi, r}$ and $\setD_{\Psi,r}$ are convex and max-convex;}
\item[(ii)] {$\setD_{\Psi}$ consists of only one vector if and only if $S^C$ is irreducible and $N^C(S)=N(S)$.}
\item[(iii)] {If all matrices in $\Psi$ are irreducible then $\setD_{\Psi, r}$ is compact.}
\end{itemize}
\end{theorem}
\begin{proof}
(i): (see \cite{SergeiVS}~Proposition~3.1) $\setC_{\Psi, r}$ is convex and max-convex, since
\begin{equation}
\setC_{\Psi,r}=\bigcap\limits_{i,j}\{x\in\operatorname{int}(\RR_+^n)\mid s_{ij} x_j\leq rx_i\},
\end{equation}
where the sets whose intersection is taken are convex and max-convex. $\setD_{\Psi,r}$
inherits these properties as a coordinate section of $\setC_{\Psi, r}$.

(ii) follows from Proposition~\ref{pro:unique} applied to $S$; 
(iii) follows from~\cite{Drie4}~Theorem~1 applied to $S$. 
\end{proof}

We remark that sets that are both convex and max-convex have appeared under various names like
Kleene cones \cite{SergeiVS}, polytropes \cite{JoswigKulas}, or zones \cite{Mine}.

\section{Globally Optimal Solutions}
\label{sec:commsubevec}
The application of the max algebra to the AHP is motivated in \cite{Drie3, Drie4} by the following considerations.  First, it is observed that, for an {\it{SR}}-matrix $A$, vectors in the set $\setC_A$ minimise the function (\ref{eq:errf}) and hence the relative error.  Based on this observation, these vectors are used to construct transitive matrices to obtain an overall ranking of the alternatives in the decision process.  In light of the properties of $\setC_A$, this is justified by the fact that the transitive matrices constructed in this way are \textit{closest} to the original {\it{SR}}-matrix $A$ in the sense of the relative error.

Thus, the approach to construct a ranking vector for a single {\it{SR}}-matrix taken in \cite{Drie3, Drie4} amounts to solving the following optimisation problem.
\begin{equation}
\label{eq:OPT1}
\min\limits_{x \in \textrm{int}(\RR^n_+)} \{e_A(x)\}.
\end{equation}
In this and the following section, we are concerned with extending the above approach to the general AHP with $n$ alternatives and $m$ criteria.

Formally, we are given $m$ {\it{SR}}-matrices; one for each criterion.  Let $\Psi$ in (\ref{eq:Psi}) denote the set of these matrices.  For each $A_i\in\Psi$, there is an error function $e_{A_i}: \textrm{int}(\RR_{+}^{n})\rightarrow \RR_+$  defined as in (\ref{eq:errf}).  In contrast to the approach taken in the classical AHP, we view the construction of a ranking vector for the $m$ criteria as a multi-objective optimisation problem for the error functions $e_{A_i}$, $1 \leq i \leq m$.

To begin with, we seek a vector that simultaneously minimizes all of the functions $e_{A_i}$.  Such a vector is said to be a \emph{globally optimal solution} for the multi-objective optimisation problem.

Note that for each $A_i\in \Psi$, the set of vectors that minimise $e_{A_i}: \textrm{int}(\RR_+^{n})\rightarrow \RR_+$ is precisely $\setC_{A_i}$ \cite{Drie4}: formally,
\begin{equation}
\label{eq:murelmvec}
\setC_{A_i}=\{x \in \textrm{int}(\RR^n_+) \mid e_{A_i}(x) = \min\limits_{w \in \textrm{int}(\RR^n_+)}e_{A_i}(w)\}  ,\quad i=1, 2, ..., m.
\end{equation}
Hence, the problem of finding a vector $x \in \textrm{int}(\RR^n_+)$ that simultaneously minimises all the error functions $e_{A_i}$ amounts to determining when
$$\bigcap\limits_{i=1}^m \setC_{A_i} \neq \emptyset.$$
Equivalently, $x$ simultaneously minimises all the error functions if and only if it is a common subeigenvector of $A_i$ for all $i\in\{1, 2, ..., m\}$.  The remainder of this section is divided into two parts: we first consider the existence of common subeigenvectors for arbitrary nonnegative matrices in the next subsection; we then specialise to sets of {\it SR}-matrices and globally optimal solutions.

\subsection{Common Max-algebraic Subeigenvectors of Nonnegative Matrices}
First of all, we consider the general problem of finding a common subeigenvector for a set of nonnegative 
matrices (not necessarily {\it{SR}}-\-mat\-ri\-ces).  Our results are clearly related to the work in \cite{SergeiComm} concerning the intersection of eigencones of commuting matrices over the max and nonnegative algebra.

In the next result, we adopt the notation $\hat A_i=\frac{A_i}{\mu(A_i)}$ for $1 \leq i \leq m$ and, in an abuse of notation, $\hat S=\bigoplus\limits_{i=1}^{m}\hat A_i.$
\begin{theorem}
\label{thm:comxsh}
Consider the set $\Psi\subset \RR_+^{n\times n}$ in (\ref{eq:Psi}).  
The following
assertions are equivalent.
\begin{itemize}
\item[{(i)}] $\mu(\hat S)=1$;
\item[{(ii)}] There exists some $x\in \operatorname{int}(\RR^n_+)$ with $A_i\otimes x\leq \mu(A_i)x$ for all $A_i\in\Psi$;
\item[{(iii)}] $\mu(A_{j_1}\otimes\cdots\otimes A_{j_p}) \leq \mu(A_{j_1})\cdots\mu(A_{j_p})$ where $1 \leq j_k \leq m$ for $1 \leq k \leq p$. (We say that $\mu$ is submultiplicative on $\Psi$).
\end{itemize}
\end{theorem}
\begin{proof}
(i)$\Rightarrow$(ii): First, assume that $\mu(\hat S)=1$. Then, there exists $x\in \textrm{int}(\RR^n_+)$ such that $\hat S\otimes x\leq x$. 
Thus, $\hat A_i\otimes x\leq \hat S\otimes x\leq x$ for all $A_i\in\Psi$. Hence, $A_i\otimes x\leq \mu(A_i)x$ for all $i$.

(ii)$\Rightarrow$(iii): Suppose that there exists some $x\in \textrm{int}(\RR^n_+)$ with $A_i\otimes x\leq \mu(A_i)x$ for all $A_i\in\Psi$. Pick some $\psi\in\Psim$ such that $\psi=A_{j_1} \otimes A_{j_2}\otimes \cdots \otimes A_{j_p}$ where $1 \leq j_k \leq m\mbox{ for } 1 \leq k \leq p$. Then,
\begin{align*}
\psi\otimes x & = A_{j_1}\otimes A_{j_2}\otimes\cdots\otimes A_{j_{p}}\otimes x\displaybreak[0]\\\nonumber
& \leq \mu(A_{j_p})A_{j_1}\otimes A_{j_2}\otimes\cdots\otimes A_{j_{p-1}}\otimes x\displaybreak[0]\\\nonumber
&\vdots\displaybreak[0]\\\nonumber
& \leq \mu(A_{j_p})\mu(A_{j_{p-1}})\cdots \mu(A_{j_2}) A_{j_{1}}\otimes x\displaybreak[0]\\\nonumber
& \leq \mu(A_{j_p})\mu(A_{j_{p-1}}) \cdots \mu(A_{j_2}) \mu(A_{j_{1}}) x.\displaybreak[0]\\\nonumber
\end{align*}
Writing $r=\mu(A_{j_1})\mu(A_{j_{2}}) \cdots \mu(A_{j_{p}})$, we see that $x\in \setC_{\psi, r}$ from the definition (\ref{eq:CAr}). Hence, $\setC_{\psi, r}\neq \emptyset$.  Point (i) in Proposition \ref{pro:CArprop} implies $r\geq \mu(\psi)$. Thus, $\mu(A_{j_{1}})\mu(A_{j_2})\cdots \mu(A_{j_p})\geq \mu(A_{j_1}\otimes A_{j_2}\otimes\cdots\otimes A_{j_p})$. Note that we
essentially used~\eqref{eq:CW}.

(iii)$\Rightarrow$(i): Consider the set of normalised matrices $$\hat \Psi=\{\hat A_1, \hat A_2, ..., \hat A_m\}$$ where $\mu(\hat A_i)=1$ for all $i\in\{1, 2, ..., m\}$. Pick some $\psi\in\hPsim$.  As $\mu$ is submultiplicative on $\Psi$, it is also submultiplicative on $\hat \Psi$.  Thus, we have $\mu(\psi)\leq 1.$   As this is true for any $\psi \in \hPsim$, it follows that $$\max\limits_{\psi\in\hPsim}\mu(\psi)\leq 1.$$
Taking the $p^{\text{th}}$ root and $\lim\sup\limits_{p\rightarrow\infty}$ of both sides, we see that
$$\hmu(\hat \Psi)\leq 1.$$
Theorem \ref{thm:muMSPsiS} (ii) then implies that $\mu(\hat S)\leq 1$.  Furthermore, since $\mu(\hat S)\geq \mu(\hat A_i)=1$ for all $i\in\{1, 2, ..., m\}$ we obtain $\mu(\hat S)=1$.
\end{proof}
Note that the equivalence (i)$\Leftrightarrow$(ii) can be regarded
as a special case of Hershkowitz-Schneider~\cite{HS-03} Theorem 2.5, see also \cite{ButSchn}
Theorem 3.5 for an extension. In these works, the problem of
simultaneous nonnegative matrix scaling is considered; this amounts to finding
a diagonal matrix $X$ such that $XA_kX^{-1}\leq B_k$ for
$k=1, \ldots, m$.  For our case, take $B_k=\mathbf{1}_{n\times n}$ (the all-ones matrix) and
impose $\mu(A_k)=1$. However, condition (iii) does not appear
in \cite{HS-03} or \cite{ButSchn}.

In the terminology of Butkovi\v{c}~et~al.~\cite{ButkovicBook, SergeiVS}, there exists a {\bf
simultaneous visualisation} of all of the matrices in $\Psi$, meaning that
$X^{-1}A_iX\leq\mu(A_i)\cdot\mathbf{1}_{n\times n}$ for all $i$, and
in particular $(X^{-1}A_iX)_{kl}=\mu(A_i)$ for all $i$ and $(k,l)\in
E^C(A_i)$.
The following result for general nonnegative matrices will be useful in the next subsection to clarify the relationship 
between commutativity and the existence of globally optimal solutions for $3 \times 3$ {\it SR}-matrices.
\begin{proposition}
\label{p:commoncrit}
Let $A, B\nnegm$ have $\mu(A)=\mu(B)=1$. If $\mu(A\oplus B)=1$, then
\begin{itemize}
\item[(i)] $a_{ij}=b_{ij}$ for all edges $(i, j)\in E^C(A)\cap E^C(B)$;
\item[(ii)] $a_{ij}b_{ji}=1$ for $(i,j)\in E^C(A)$ and $(j,i)\in E^C(B)$.
\end{itemize}
\end{proposition}
\begin{proof}
As, $\mu(A)=\mu(B)=1$, it follows that $\hat A=A$ and $\hat B=B$.  Thus, $\hat S=A\oplus B$. From the assumption, we obtain $$\mu(\hat S) = 1.$$  It now follows from Theorem \ref{thm:comxsh} that there exists some $x > 0$ with $A \otimes x \leq x$, $B \otimes x \leq x$.  Let $X=\diag(x)$ and consider the diagonally scaled matrices
$$X^{-1}AX, \;\;\; X^{-1}BX.$$
From the choice of $X$ it is immediate that
\begin{equation}
\label{eq:vis1}
X^{-1}AX \leq \mathbf{1}_{n \times n}, \;\;\; X^{-1}BX \leq \mathbf{1}_{n \times n}.
\end{equation}
Furthermore, from well-known facts about diagonal scaling (see Proposition 2.10 of \cite{SergeiVS}), it follows that
\begin{equation}
\label{eq:mu1}
1 = \mu(A) = \mu(X^{-1}AX), \;\; 1 = \mu(B) = \mu(X^{-1}BX)
\end{equation}
and that
\begin{equation}
\label{eq:Ec1}
E^C(A) = E^C(X^{-1}AX), \;\; E^C(B) = E^C(X^{-1}BX).
\end{equation}

We prove (i): Let $(i, j) \in E^C(A)\cap E^C(B)$ be given.  It follows from (\ref{eq:Ec1}) that $(i, j)$ is also a critical edge in the digraphs of $X^{-1}AX$ and $X^{-1}BX$.  (\ref{eq:vis1}) and (\ref{eq:mu1}) now imply that
$$\frac{a_{ij}x_j}{x_i} = \frac{b_{ij}x_j}{x_i} = 1.$$
Hence
$$a_{ij} = b_{ij} = \frac{x_i}{x_j}$$
and this completes the proof.

We prove (ii): Let $(i,j)\in E^C(A)$ and $(j,i)\in E^C(B)$. It follows from (\ref{eq:Ec1}) that $(i,j)$ is also a critical edge
in the digraph of $X^{-1}AX$ and $(j,i)$ is a critical edge in the digraph of $X^{-1}BX$. Then
$$\frac{a_{ij}x_j}{x_i}=\frac{b_{ji}x_i}{x_j}=1,$$
and hence
$$a_{ij}b_{ji}=\frac{a_{ij}x_j}{x_i}\cdot \frac{b_{ji}x_i}{x_j}=1.$$
\end{proof}

We next recall the following result, which was established in \cite{SergeiComm} and shows that commutativity is a sufficient condition for the existence of a common eigenvector for irreducible matrices.
\begin{proposition}[\cite{SergeiComm}]
\label{pro:conds}
Consider the set $\Psi\subset \RR_+^{n\times n}$ in (\ref{eq:Psi}).  Assume that each $A_i \in \Psi$ is irreducible and moreover that
\begin{equation}
\label{eq:comm1} A_i\otimes A_j = A_j\otimes A_i \;\; \mbox{ for } 1 \leq i, j \leq m.
\end{equation}
Then there exists some $x \in \operatorname{int}(\RR^n_+)$ with $A_i\otimes x = \mu(A_i)x$ for $1 \leq i \leq m$.
\end{proposition}

The next corollary is an immediate consequence of Proposition \ref{pro:conds} and the fact, which we recalled in Section \ref{sec:intro}, that for an irreducible matrix $A$, the set $\setC_A$ is the subeigencone $V^*(A)$,
which coincides with the eigencone $V(\hat A^*)$.

\begin{corollary}
\label{cor:Astar}
Consider the set $\Psi\subset \RR_+^{n\times n}$ in (\ref{eq:Psi}).  Assume that each $A_i \in \Psi$ is irreducible and moreover that
\begin{equation}
\label{eq:comm2}
\hat A_i^*\otimes \hat A_j^* = \hat A_j^*\otimes \hat A_i^* \;\; \mbox{ for } 1 \leq i, j \leq m.
\end{equation}
Then there exists some $x\in \operatorname{int}(\RR^n_+)$ with $A_i\otimes x \leq \mu(A_i)x$ for $1 \leq i \leq m$.
\end{corollary}

Note that~\eqref{eq:comm1} implies~\eqref{eq:comm2}.

\subsection{{\it SR}-matrices and Globally Optimal Solutions}
In the remainder of this section, we will only focus on {\it{SR}}-matrices.  We first present the following corollary of Theorem \ref{thm:comxsh}, which develops the concept of simultaneous visualization for {\it{SR}}-matrices.   Before stating the corollary, define the {\bf anticritical graph} of an {\it{SR}}-matrix to consist of the edges $E^{\overline{C}}(A)$ given by:
\begin{equation}
\label{def:anticrit}
(i,j)\in E^{\overline{C}}(A)\Leftrightarrow (j,i)\in E^C(A)
\end{equation}

\begin{corollary}
\label{c:comxsh-sr}
Consider the set $\Psi\subset \RR_+^{n\times n}$ in (\ref{eq:Psi}).  Assume that each $A_i \in \Psi$ is an {\it{SR}}-matrix.
If any of the equivalent statements of Theorem \ref{thm:comxsh} holds, then
there exists some $x\in \textrm{int}(\RR^n_+)$ such that for $X = \diag(x)$ we have
\begin{equation}
\label{eq:sandwich}
\mu^{-1}(A_i)\cdot{\mathbf 1}_{n\times n}\leq X^{-1}A_iX\leq\mu(A_i)\cdot {\mathbf 1}_{n\times n}
\end{equation}
In particular,
\begin{equation}
\label{eq:CritNonC}
\begin{split}
(k,l)\in E^C(A_i)\Leftrightarrow (X^{-1}A_iX)_{kl}=\mu(A_i),\\
(k,l)\in E^{\overline{C}}(A_i)\Leftrightarrow (X^{-1}A_iX)_{kl}=\mu^{-1}(A_i),
\end{split}
\end{equation}
\end{corollary}
\begin{proof}
The right-hand side inequality of (\ref{eq:sandwich}) is the same as
Theorem \ref{thm:comxsh} (ii). For the remaining left-hand side inequality
of (\ref{eq:sandwich}) we observe that $x_i^{-1}a_{ij}x_j\leq\mu(A)$ is equivalent to
$x_j^{-1}a^{-1}_{ij}x_i\geq\mu^{-1}(A)$. Then we apply $a^{-1}_{ij}=a_{ji}$.
\end{proof}

We next show that two distinct {\it{SR}}-matrices $A, B$ in $\mathbb{R}_+^{2 \times 2}$ cannot have a common subeigenvector.
Let
$$A=\left[ \begin{array}{cc}
        1   & a  \\
        1/a & 1  \\
\end{array} \right], B=\left[ \begin{array}{cc}
        1   & b  \\
        1/b & 1  \\
\end{array} \right]$$
and assume that $A \neq B$.
Clearly, $\mu(A)=\mu(B)=1$ and $\hat S = A\oplus B$.  If $a>b$, then $1/a<1/b$ and $\mu(\hat S)=a/b>1$. If $b>a$, then $1/b<1/a$ and $\mu(\hat S)=b/a>1$. In both cases, $\mu(\hat S)\neq 1.$  Hence by Theorem \ref{thm:comxsh}, $A$ and $B$ do not have a common subeigenvector.

Proposition \ref{pro:conds} shows that commuting irreducible matrices possess a common max eigenvector.  We now show that for $3 \times 3$ {\it{SR}}-matrices, commutativity is both necessary and sufficient for the existence of a common subeigenvector.

\begin{remark}
\label{rem:3by3set}
For an {\it {SR}}-matrix $A\in\RR_+^{3 \times 3}$, it is immediate that all cycle products of length one and two in $D(A)$ are equal to 1.  Further, there are two possible cycle products of length 3 in $D(A)$: $a_{12}a_{23}a_{31}$ and $a_{13}a_{32}a_{21}$.  As $A$ is an {\it {SR}}-matrix, it follows that
$$a_{12}a_{23}a_{31} = \frac{1}{a_{13}a_{32}a_{21}}$$
and hence at least one of the above products must be greater than or equal to 1.  Since $\mu(A)\geq 1$, one of the cycles of length three is critical, and the other cycle is anticritical.  Thus, $N^C(A)=N(A)$ and $A^C$ is irreducible.  Hence, it follows from Proposition \ref{pro:unique} that $A$ has a unique subeigenvector up to a scalar multiple in $\setC_A$ which is its max eigenvector.
Observe that each edge $(i,j)$ with $i\neq j$ belongs either to the critical or to the anticritical graph.
\end{remark}

Our next result characterises when two {\it {SR}}-matrices in $\RR_+^{3 \times 3}$ have a common subeigenvector.

\begin{theorem}
\label{thm:3b3cond} Let a set $\{A, B\}\subset\RR_+^{3\times 3}$ of {\it{SR}}-matrices be given.  Write $\hat S=\hat A\oplus \hat B$. The following are equivalent.
\begin{itemize}
\item[{(i)}] $\mu(\hat S)=1$;
\item[{(ii)}] $A$ and $B$ commute;
\item[{(iii)}] There exists a vector $x \in \textrm{int}(\RR^n_+)$ with $A \otimes x \leq \mu(A) x$, $B \otimes x \leq \mu(B) x$.
\end{itemize}
\end{theorem}
\begin{proof}
The equivalence of (i) and (iii) follows immediately from Theorem \ref{thm:comxsh} so we will show that (i) and (ii) are also equivalent.

(ii) $\Rightarrow$ (i) follows immediately from Proposition \ref{pro:conds}.  We prove (iii)$\Rightarrow$(ii).
First note that it follows from Remark \ref{rem:3by3set} that for distinct $i, j, k$,  the edges $(i, j)$, $(j, k)$ are either both critical or both anti-critical for $A$.  The same is true of $B$.  Calculating $X^{-1}AX$ and $X^{-1}BX$ where $X=\diag(x)$, it follows from Theorem \ref{thm:comxsh}, Corollary \ref{c:comxsh-sr} and the identities $\mu(A)\mu(B)=\mu(B)\mu(A)$, $\mu(A)\mu^{-1}(B)=\mu^{-1}(B)\mu(A)$ that
\begin{equation}
\label{eq:OlComm}
a_{ij}b_{jk} = b_{ij}a_{jk}
\end{equation}
for any distinct $i,j,k$.  
It now follows from (\ref{eq:OlComm}) that for $i \neq j$
\begin{align*}
(A\otimes B)_{ij} & = a_{ii}b_{ij}\oplus a_{ij}b_{jj}\oplus a_{ik}b_{kj} \displaybreak[0]\\ \nonumber
 & = b_{ij}\oplus a_{ij}\oplus b_{ik}a_{kj} \displaybreak[0]\\ \nonumber
 & = (B\otimes A)_{ij} \displaybreak[0]\\ \nonumber
\end{align*}
where $k \neq i$, $k \neq j$.  Rewriting (\ref{eq:OlComm}) as $a_{ji}b_{ik} = b_{ji}a_{ik}$, it follows readily that $b_{ik}a_{ki} = a_{ij}b_{ji}$ and $a_{ik}b_{ki} = b_{ij}a_{ji}$.  It now follows that for $1 \leq i \leq 3$,
\begin{align*}
(A\otimes B)_{ii} & = a_{ii}b_{ii}\oplus a_{ij}b_{ji}\oplus a_{ik}b_{ki} \displaybreak[0]\\ \nonumber
 & = b_{ii}a_{ii}\oplus b_{ij}a_{ji}\oplus b_{ik}a_{ki} \displaybreak[0]\\ \nonumber
 & = (B\otimes A)_{ii} \displaybreak[0]\\ \nonumber
\end{align*}
Thus, $A \otimes B = B \otimes A$ as claimed.
\end{proof}

It is now straightforward to extend the above result to an arbitrary finite set of { \it{SR}}-matrices in $\RR_+^{3 \times 3}$.
\begin{theorem}
\label{thm:3by3finite}
Let a set $\{A_1, \ldots , A_m\}\subset\RR_+^{3\times 3}$ of {\it{SR}}-matrices be given.  Write $\hat S=\hat A_1 \oplus \cdots \oplus \hat A_m$. The following are equivalent.
\begin{itemize}
\item[{(i)}] $\mu(\hat S)=1$;
\item[{(ii)}] $A_i \otimes A_j = A_j \otimes A_i$ for all $i, j$;
\item[{(iii)}] There exists a vector $x \in \textrm{int}(\RR^n_+)$ with $A_i \otimes x \leq \mu(A_i) x$ for all $i$.
\end{itemize}
\end{theorem}
\begin{proof}
As above, the equivalence of (i) and (iii) follows immediately from Theorem \ref{thm:comxsh} and (ii) $\Rightarrow$ (i) follows immediately from Proposition \ref{pro:conds}.  To show that (i) $\Rightarrow$ (ii), suppose $\mu(\hat S) =1 $.  Then it follows that for all $i, j$, $$\hat A_i \oplus \hat A_j \leq \hat S$$
and hence that $\mu(\hat A_i \oplus \hat A_j) \leq 1$.  As $\mu (\hat A_i \oplus \hat A_j) \geq 1$, it is immediate that $$\mu(\hat A_i \oplus \hat A_j) = 1$$ for all $i, j$ in $\{1, \ldots, m\}$.  It follows immediately from Theorem \ref{thm:3b3cond} that $$A_i \otimes A_j = A_j \otimes A_i$$ for $1 \leq i, j \leq m$ as claimed.
\end{proof}

We note with the following example that commutativity is not a necessary condition for $4\times 4$ {\it {SR}}-matrices to possess a common subeigenvector.
\begin{example}
\label{ex:ex4by4}
Consider the {\it{SR}}-matrices given by
$$A=\left [ \begin{array}{cccc}
   1 & 8 & 1/4 & 7 \\
   1/8 & 1 & 6 & 1/4 \\
   4 & 1/6 & 1 & 4 \\
   1/7 & 4 & 1/4 & 1 \\
\end{array} \right ] B=\left [ \begin{array}{cccc}
   1 & 4 & 5 & 9 \\
   1/4 & 1 & 1/8 & 9 \\
   1/5 & 8 & 1 & 1/8 \\
   1/9 & 1/9 & 8 & 1 \\
\end{array} \right ]$$
where $\mu(\hat S)=1$. Here, $x=\left [ \begin{array}{cccc}
   1 &
   0.721 &
   0.693 &
   0.667 \\
\end{array} \right ]^T$ is a common subeigenvector.  However, it can be readily verified that $A \otimes B \neq B \otimes A$.
\end{example}
\section{Min-max Optimal Points and the Generalised Spectral Radius}
\label{sec:minmax}
In general, it will not be possible to find a single vector $x$ that is globally optimal for the set $\Psi$ of {\it SR}-matrices given by (\ref{eq:Psi}).  With this in mind, in this short section we consider a different notion of optimal solution for the multiple objective functions $e_{A_i}: \textrm{int}(\RR_{+}^{n})\rightarrow \RR_+$, $1 \leq i \leq m$.  In fact, we consider the following optimisation problem.
\begin{equation}
\label{eq:OPT2}
\min\limits_{x \in \textrm{int}(\RR^n_+)} \left(\max\limits_{1 \leq i \leq m} e_{A_i}(x) \right).
\end{equation}
In words, we are seeking a weight vector that minimises the maximal relative error where the maximum is taken over the $m$ criteria ({\it{SR}}-matrices).

Corollary \ref{cor:CPrprop} has the following interpretation in terms of the optimisation problem given in (\ref{eq:OPT2}).
\begin{proposition}
\label{pro:OPT2} Consider the set $\Psi$ given by (\ref{eq:Psi}). Then: 
\begin{itemize}
\item[{(i)}] $\hmu(\Psi) = \min\limits_{x\in\operatorname{int}(\RR_+^n)} \left(\max\limits_{1 \leq i \leq m} e_{A_i}(x) \right)$;
\item[{(ii)}] $x$ solves (\ref{eq:OPT2}) if and only if $x \in \setC_{\Psi}$.
\end{itemize}
\end{proposition}
\begin{proof}
Corollary~\ref{cor:CPrprop} shows that there exists some $x \in \textrm{int}(\RR^n_+)$ with
$$\max\limits_{1 \leq i \leq m} e_{A_i}(x) \leq r$$
if and only if $r \geq \hmu(\Psi)$.   (i) follows from this observation.  The result of (ii) 
is then immediate from the definition of $\setC_{\Psi}$.
\end{proof}

\section{Pareto Optimality and the AHP}
\vskip0.2cm
Thus far, we have considered two different approaches to the multi-objective optimisation problem associated with the AHP.  In this section we turn our attention to what is arguably the most common framework adopted in multi-objective optimisation: Pareto Optimality \cite{Bewley, Miettinen}.  As above, we are concerned with the existence of optimal points for the set of objective functions $e_{A_i}$, for $1 \leq i \leq m$ associated with the set $\Psi$ (\ref{eq:Psi}) of {\it{SR}}-matrices. 
We first recall the notion of \textit{weak Pareto optimality}.

\begin{definition}[\cite{Miettinen}]
\label{def:WPOP}
 $w \in \operatorname{int}(\RR^n_+)$ is said to be a {\it{weak Pareto optimal point}} for the functions $e_{A_i}: \textrm{int}(\RR_+^{n})\rightarrow \RR_+ (1\leq i\leq m)$ if there does not exist $x \in \operatorname{int}(\RR^n_+)$ such that
$$e_{A_i}(x)<e_{A_i}(w)$$ for all $i=1, 2, ..., m$.
\end{definition}
The next lemma shows that every point in the set $\setC_{\Psi}$ is a weak Pareto optimal point for $e_{A_1}, \ldots , e_{A_m}$.
\begin{lemma}
\label{lem:WPOP}
Let $\Psi\subset\RR_{+}^{n\times n}$ be given by (\ref{eq:Psi}).
Any $w\in \setC_{\Psi}$ is a weak Pareto optimal point for $e_{A_1}, \ldots , e_{A_m}$.
\end{lemma}
\begin{proof}
Let $w \in \setC_{\Psi}$ be given.  Then $e_{A_i}(w) \leq \hmu(\Psi)$ for $1 \leq i\leq m$.  If there exists some $x \in \textrm{int}(\RR^n_+)$ such that $e_{A_i}(x)<e_{A_i}(w)$ for $1 \leq i \leq m$, then for this $x$
$$e_{A_i}(x)<\hmu(\Psi)$$ for $1 \leq i \leq m$.  This contradicts Proposition \ref{pro:OPT2}.
\end{proof}

We next recall the usual definition of a Pareto optimal point.
\begin{definition}[\cite{Miettinen}]
\label{def:POP}
 $w \in \operatorname{int}(\RR^n_+)$ is said to be a {\it{Pareto optimal point}} for the functions $e_{A_i}: \operatorname{int}(\RR^n_+) \rightarrow \RR_+ (1\leq i\leq m)$ if $e_{A_i}(x)\leq e_{A_i}(w)$ for $1 \leq i \leq m$ implies $e_{A_i}(x)=e_{A_i}(w)$ for all $1 \leq i \leq m$.
\end{definition}
We later show that the multi-objective optimisation problem associated with the AHP always admits a Pareto optimal point.  We first present some simple facts concerning such points.
\begin{theorem}
\label{thm:pops}
Let $\Psi\subset\RR_{+}^{n\times n}$ be given by (\ref{eq:Psi}).  Then:
\begin{itemize}
\item[{(i)}] If $w\in \setC_{\Psi}$ is unique up to a scalar multiple, then it is a Pareto optimal point for $e_{A_1}, \ldots , e_{A_m}$;
\item[{(ii)}] If $w\in \setC_{A_i}$ is unique up to a scalar multiple for some $i\in\{1, 2, ..., m\}$, then it is a Pareto optimal point for $e_{A_1}, \ldots , e_{A_m}$.
\end{itemize}
\end{theorem}
\begin{proof}
Observe that both conditions imply $\hat\mu(\Psi)>0$.

(i) Assume that $w\in \setC_{\Psi}$ is unique up to a scalar multiple. Pick some $x\in \textrm{int}(\RR^n_+)$ such that $e_{A_i}(x)\leq e_{A_i}(w)$ for all $i$. Then, $e_{A_i}(x)\leq\hat\mu(\Psi)$ for all $i$ which implies that $x\in \setC_\Psi$. Thus, $x=\alpha w$ for some $\alpha\in\RR_+$. Hence, $e_{A_i}(x)=e_{A_i}(w)$ for all $i$ and $w$ is a Pareto optimal point.

(ii) Assume that for some $A_i\in\Psi$, $w\in \setC_{A_i}$ is unique up to a scalar multiple.  
\if{
Since $A_i\neq 0$ (otherwise $\setC_{A_i}$ contains all positive vectors),
$\setC_{A_i}$ being non-empty implies $\mu(A_i)>0$, and further $\mu(S)>0$ and
$\setC_S\neq\emptyset$. Hence $\setC_{\Psi}=\cap_{i=1}^m \setC(A_i)= \setC_S$ 
contains at least one positive vector with all of its multiples, and this vector is nothing else but $w$.
It remains to apply (i).
}\fi
Suppose $x\in \textrm{int}(\RR^n_+)$ is such that $e_{A_j}(x)\leq e_{A_j}(w)$ for all $1 \leq j \leq m$.  In particular, $x \in \setC_{A_i}$, and this implies that $x=\alpha w$ for some $\alpha\in\RR$. Further, it is immediate that for any other $A_j\in \Psi$ ($i\neq j$), we have $e_{A_j}(x)=e_{A_j}(w)$. Thus, $w$ is a Pareto optimal point.
\end{proof}

By Proposition~\ref{pro:unique}, condition~(i) is equivalent to $N^C(S)=N(S)$ and $S^C$ to be irreducible, and condition~(ii) is equivalent to $N^C(A_i)=N(A_i)$ and $A_i^C$ to be irreducible for some $i$.

\begin{corollary}
\label{cor:pops}
Let the set $\Psi\subset \RR_{+}^{n\times n}$ given by (\ref{eq:Psi}) consist of {\it{SR}}-matrices. For $n\in \{2, 3\}$, any $w\in \setC_{A_i} (1\leq i\leq m)$ is a Pareto optimal point for $e_{A_1}, \ldots , e_{A_m}$.
\end{corollary}
\begin{proof}
Notice that from Remark \ref{rem:3by3set} for $3\times 3$ case, there exists a unique subeigenvector up to a scalar multiple in each $\setC_{A_i}$ for $1\leq i\leq m$.  This is also true for the $2\times 2$ case because $N^C(A)=N(A)$ and $A^C$ is irreducible. The result directly follows from (ii) in Theorem \ref{thm:pops}.
\end{proof}

The following example demonstrates point (i) in Theorem \ref{thm:pops}.
\begin{example}
\label{ex:popsDPU}
Consider the following matrices given by
$$A=\left [ \begin{array}{cccc}
   1 & 9 & 1/4 & 2 \\
   1/9 & 1 & 6 & 3 \\
   4 & 1/6 & 1 & 1/4 \\
   1/2 & 1/3 & 4 & 1 \\
\end{array} \right ] B=\left [ \begin{array}{cccc}
   1 & 1/2 & 4 & 1/8 \\
   2 & 1 & 3 & 2 \\
   1/4 & 1/3 & 1 & 5 \\
   8 & 1/2 & 1/5 & 1 \\
\end{array} \right ].$$ $S$ matrix is obtained as follows
$$S=\left [ \begin{array}{cccc}
   1 & 9 & 4 & 2 \\
   2 & 1 & 6 & 3 \\
   4 & 1/3 & 1 & 5 \\
   8 & 1/2 & 4 & 1 \\
\end{array} \right ]$$
where $N^C(S)=N(S)$ and $S^C$ is irreducible. From Proposition~\ref{pro:unique}, we have a unique vector (up to a scalar multiple) in $\setC_{\Psi}$: $w=\left [ \begin{array}{cccc} 1 & 0.758 & 0.861 & 1.174 \\ \end{array} \right ]^T$. Figure \ref{fig:ex1} below represents the values of $e_{A}(x)$ and $e_{B}(x)$ at $w$ and some points in $C_A$, $C_B$ and $\textrm{int}(\RR^n_+)$. Remark that Pareto optimality is observed at $w\in C_{\Psi}$  where $e_{A}(w)=e_{B}(w)=6.817$.
\begin{figure}[h!]
\centering
\includegraphics[width=0.8\textwidth]{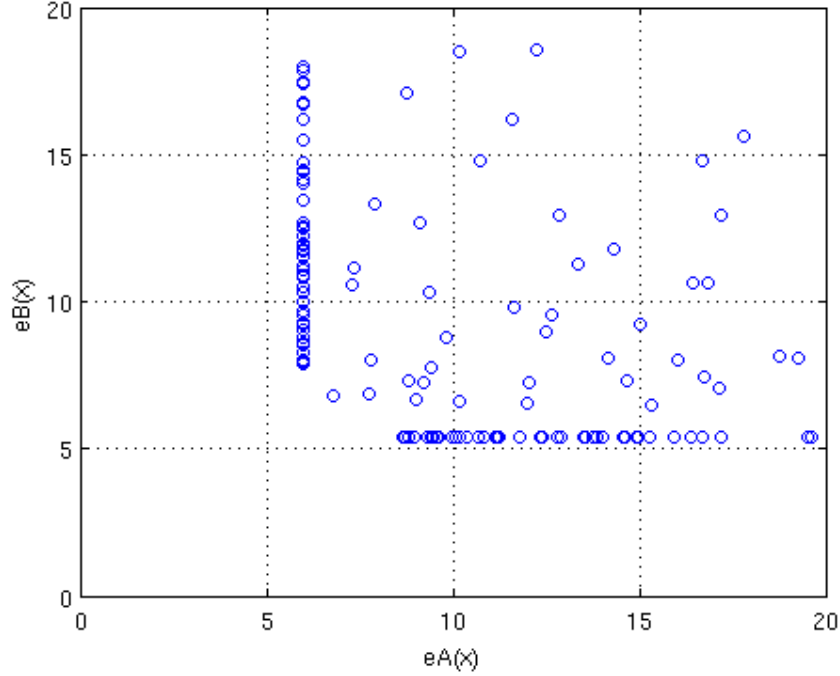}
\caption[Refer to Example \ref{ex:popsDPU}]{Refer to Example \ref{ex:popsDPU}}
\label{fig:ex1}
\end{figure}
\end{example}
Our objective in the remainder of this section is to show that the multi-objective optimisation problem associated with the AHP \textit{always} admits a Pareto optimal solution.  We first recall the following general result giving a sufficient condition for the existence of a Pareto optimal point with respect to a set $E\subset\RR^{n}$.  Essentially, this is a direct application of the fact that a continuous function on a compact set always attains its minimum.
\begin{theorem}[\cite{Bewley}]
\label{thm:existence}
Let $E \subseteq \RR^n$ be nonempty and compact.  Let a set of continuous functions $\{f_1,\ldots , f_m\}$ be given where
$$f_i: E \rightarrow \RR_+$$ for $1 \leq i \leq m$.
There exists $w \in E$ such that $x \in E$, $f_i(x) \leq f_i(w)$ for $1 \leq i \leq m$ implies $f_i(x) = f_i(w)$ for $1 \leq i \leq m$.
\end{theorem}
This result follows from elementary real analysis and the observation that if $w$ minimises the (continuous) weighted sum $\sum\limits_{i=1}^m \alpha_i f_i(x)$ where $\alpha_i > 0$ for $1 \leq i \leq m$, then $w$ must be Pareto optimal for the functions $f_1, \ldots, f_m$.  

A point $w$ satisfying the conclusion of the Theorem \ref{thm:existence} is said to be Pareto optimal for $\{f_1, \ldots, f_m\}$ {\it with respect to }$E$.  Thus, for any multi-objective optimisation problem with continuous objective functions defined on a compact set, there exists a point that is Pareto optimal with respect to the given set.

To apply the above result to the AHP, we first note that for a set $\Psi$ of {\it{SR}}-matrices, $\setD_{\Psi}$ is compact by Theorem \ref{thm:geop} (recall that {\it{SR}}-matrices are positive and hence irreducible).  
We now show that any point in $\setD_{\Psi}$ that is Pareto optimal with respect to $\setD_{\Psi}$ is also Pareto optimal with respect to the set $\textrm{int}(\RR^n_+)$.  

Although we don't specifically use the {\it{SR}} property in the following, we assume $\Psi$ to consist of {\it{SR}}-matrices as we are primarily interested in the AHP application. Instead, we could just assume that $A_i$ are irreducible implying that
$\setD_{\Psi}$ is compact and a Pareto optimal point exists.
\begin{lemma}
\label{lem:existence}
Consider the set $\Psi$ in (\ref{eq:Psi}) and assume that $A_i$ is an {\it{SR}}-matrix for $1 \leq i \leq m$. Let $w$ be a Pareto optimal point for $\{e_{A_1}, \ldots, e_{A_m}\}$ with respect to $\setD_{\Psi}$. Then, $w$ is also Pareto optimal for $e_{A_1}, \ldots , e_{A_m}$ with respect to $\textrm{int}(\RR^n_+)$.
\end{lemma}
\begin{proof}
Assume that $w\in \setD_{\Psi}$ is a Pareto optimal point with respect to $\setD_{\Psi}$.  Suppose $x \in \textrm{int}(\RR^n_+) \backslash \setC_{\Psi}$.  Then from the definition of $\setC_{\Psi}$ (\ref{eq:CPr}), it follows that
\begin{equation}
\label{eq:Par1}
e_{A_{i_0}}(x) > \hmu(\Psi) \mbox{ for some } i_0.
\end{equation}
As $w \in \setD_{\Psi}$, $e_{A_i}(w) \leq \hmu(\Psi)$ for $1 \leq i \leq m$.  It follows immediately from (\ref{eq:Par1}) that for any $x \notin \setC_{\Psi}$, it cannot happen that $e_{A_i}(x) \leq e_{A_i}(w)$ for $1 \leq i \leq m$.

Let $x$ in $\setC_{\Psi}$ be such that
$$e_{A_i}(x) \leq e_{A_i}(w) \;\; \mbox{ for } 1 \leq i \leq m.$$  As $e_{A_i}(\lambda x) = e_{A_i}(x)$ for all $\lambda > 0$, $1 \leq i \leq m$, and $w$ is Pareto optimal with respect to $\setD_{\Psi}$, it follows that $e_{A_i}(x) = e_{A_i}(w)$ for $1 \leq i \leq m$.
\end{proof}

Our next step is to show that there exists a point $x \in \setD_{\Psi}$ that is Pareto optimal with respect to $\setD_{\Psi}$.
\begin{proposition}
\label{pro:existence}
Consider the set $\Psi$ in (\ref{eq:Psi}) and assume that $A_i$ is an {\it{SR}}-matrix for $1 \leq i \leq m$.  There exists $x \in \setD_{\Psi}$ that is Pareto optimal for $e_{A_1}, \ldots, e_{A_m}$ with respect to $\setD_{\Psi}$.
\end{proposition}
\begin{proof}
First note that $\setD_{\Psi}\neq \emptyset$ since $\hat \mu(\Psi)>0$.  Theorem \ref{thm:geop} shows that $\setD_{\Psi}$ is compact.  Furthermore, for any irreducible matrix $A$, the function $e_A:\setD_{\Psi} \rightarrow \RR_+$ is a composition of continuous functions and hence continuous on a compact set.  Theorem \ref{thm:existence} implies that there exists $w$ in $\setD_{\Psi}$ that is Pareto optimal with respect to $\setD_{\Psi}$.
\end{proof}
Combining Proposition \ref{pro:existence} with Lemma \ref{lem:existence}, we immediately obtain the following result.
\begin{corollary}
\label{cor:existence}
Consider the set $\Psi$ in (\ref{eq:Psi}) and assume that $A_i$ is an {\it{SR}}-matrix for $1 \leq i \leq m$.  There exists $x \in \setD_{\Psi}$ that is Pareto optimal for $e_{A_1}, \ldots, e_{A_m}$ with respect to $\textrm{int}(\RR^n_+)$.
\end{corollary}
Corollary~\ref{cor:existence} means that there exists a vector $x$ of positive weights that is simultaneously Pareto optimal and also optimal in the \emph{min-max} sense of Section \ref{sec:minmax} for the error functions $e_{A_i}$, $1 \leq i \leq m$. 

Finally, to illustrate the above results, we revisit Example \ref{ex:Saaty}. 
\begin{example}[Example \ref{ex:Saaty} Revisited]
\label{ex:Saaty2}
Let $C, A_1, \ldots, A_5$ be as in Example \ref{ex:Saaty}.  Taking $\alpha_i$, $1 \leq i \leq 5$ to be the $i$th entry of the max eigenvector of $C$, normalised so that $\alpha_1 = 1$, we apply Theorem \ref{thm:existence} to compute Pareto optimal solutions in the set $\setD_{\Psi}$ by minimising the weighted sum 
$$\sum\limits_{i=1}^{m} \alpha_i e_{A_i}(x)$$
using the MATLAB function {{\it fminsearch}}.

Observe that $\mu(\hat S)=4.985$, so there is no common subeigenvector in this case. Next, we calculate the max eigenvector of $C:$ $$\left [ \begin{array}{c} 1 \\ 1.495 \\ 2.236 \\ 3.344\\ 0.897 \\ \end{array} \right ]$$  We find that there are multiple Pareto optimal points giving at least two possible distinct rankings: $1 > 3 > 4 > 2$ and $1 > 3 > 2 > 4$. Notice that first ranking scheme is the same as the one obtained from the classical method used in Example \ref{ex:Saaty}.  The second ranking scheme is also reasonable, since if we analyse the local rankings associated with the set of {\it{SR}}-matrices in detail, we see that $2>4$ for $A_1$, $A_2$ and $A_5$. In particular, $2$ is preferred to all other alternatives for $A_2$. 
\end{example}
\section{Conclusions and Future Work}
\label{sec:conclus}
Building on the work of Elsner and van~den~Driessche~\cite{Drie3, Drie4}, we have considered a max-algebraic approach to the multi-criteria AHP within the framework of multi-objective optimisation.  Papers~\cite{Drie3, Drie4} characterise the max eigenvectors and subeigenvectors of a \emph{single} {\it SR}-matrix as solving an optimisation problem with a single objective.  We have extended this work to the multi-criteria AHP by directly considering several natural extensions of this basic optimisation problem to the multiple objective case.  Specifically, we have presented results concerning the existence of: globally optimal solutions; min-max optimal solutions; Pareto optimal solutions.  The principal contribution of the paper is to draw attention to this max-algebraic perspective on the multi-criteria AHP, with the main results in this direction being: establishing the connection between the generalised spectral radius and min-max optimal solutions (Proposition \ref{pro:OPT2}); proving the existence of Pareto optimal solutions and showing that it is possible to simultaneously solve the Pareto and min-max optimisation problems (Proposition \ref{pro:existence} and Corollary \ref{cor:existence}).  We have also related the existence of globally optimal solutions to the existence of common subeigenvectors and highlighted connections between this question and commutativity (Theorem \ref{thm:3b3cond}).  

Buket Benek Gursoy and Oliver Mason acknowledge the support of Irish Higher Educational Authority (HEA)
PRTLI Network Mathematics Grant; Serge\u{\i} Sergeev is supported by EPSRC Grant RRAH15735, RFBR-CNRS Grant 11-01-93106 and
RFBR Grant 12-01-00886.



\begin{thebibliography}{00}
{
\bibitem{SandL} F. Baccelli, G. Cohen, G. J. Olsder, J.-P. Quadrat, Synchronization
and Linearity: An Algebra for Discrete Event Systems, John Wiley \& Sons, Chichester, New York (1992).

\bibitem{Bapat} R. B. Bapat, A max version of the Perron-Frobenius theorem, Linear
Algebra Appl. 275-276 (1998) 3-18.

\bibitem{Paper2} B. Benek Gursoy, O. Mason, {$P^1_{\max}$} and {$S_{\max}$} properties and asymptotic stability in the max algebra, Linear Algebra Appl. 435 (2011) 1008-1018.

\bibitem{Bewley} T. F. Bewley, General Equilibrium, Overlapping Generations Models, and Optimal Growth Theory, Harvard University Press, Cambridge, MA (2007).

\bibitem{ButkovicBook} P. Butkovi\v{c}, Max-linear systems: Theory and Algorithms, Springer-Verlag, London (2010).

\bibitem{ButSchn} P. Butkovi\v{c}, H. Schneider, Applications of max algebra to diagonal scaling of matrices, Elec. J. of Linear Algebra 13 (2005) 262-273.

\bibitem{CGreen} R. A. Cuninghame-Green, Minimax Algebra, Lecture Notes in Econ. and Math. Systems, 166, Springer-Verlag, Berlin (1979).

\bibitem{Drie1} L. Elsner, P. van den Driessche, On the power method in max algebra,
Linear Algebra Appl. 302-303 (1999) 17-32.

\bibitem{Drie2} L. Elsner, P. van den Driessche, Modifying the power method in max
algebra, Linear Algebra Appl. 332-334 (2001) 3-13.

\bibitem{Drie3} L. Elsner, P. van den Driessche, Max-algebra and pairwise comparison matrices, Linear Algebra Appl. 385 (2004) 47-62.

\bibitem{Drie4} L. Elsner, P. van den Driessche, Max-algebra and pairwise comparison matrices II, Linear Algebra Appl. 432 (2010) 927-935.

\bibitem{Farkas} A. Farkas, P. Lancaster, P. R{\'o}zsa, Consistency adjustment for pairwise comparison matrices, Numer. Linear Algebra Appl. 10 (2003) 689-700.

\bibitem{GaubertThes} S. Gaubert, Th\'eorie des Syst\`emes Lin\'eaires dans les Dio\"{\i}des, Ph.D. Thesis, L'\'Ecole Nationale Sup\'erieure des Mines de Paris, France (1992).

\bibitem{GaubertPerf} S. Gaubert, Performance evaluation of {$(\max,+)$} automata, IEEE
Trans. Auto. Control 40 (1995) 2014-2025.

\bibitem{GaubertRopt} S. Gaubert, Resource optimization and {$(min, +)$} spectral theory, IEEE Trans. Automat. Control 40 (1995) 1931-1934.


\bibitem{MPatWork} B. Heidergott, G. J. Olsder, J. van der Wounde, Max Plus at Work:
Modeling and Analysis of Synchronized Systems, A Course on Max-Plus Algebra and Its Applications, Princeton University Press, Princeton, NJ (2006).

\bibitem{HS-03} D. Hershkowitz, H. Schneider, One-sided simultaneous inequalities and sandwich theorems for diagonal similarity and diagonal equivalence of nonnegative matrices, Elec. J. of Linear Algebra 10 (2003) 81-101.

\bibitem{JoswigKulas} M. Joswig, K. Kulas, Tropical and ordinary convexity combined,
Adv. in Geometry 10 (2010) 333-352.

\bibitem{SergeiComm} R. D. Katz, H. Schneider, S. Sergeev, On commuting matrices in max algebra and in classical nonnegative algebra, Linear Algebra Appl. 436 (2012) 276-292.

\bibitem{Kriv05} N.K. Krivulin, Evaluation of Bounds on the Mean Rate of Growth of the State Vector of a Linear Dynamical Stochastic System in Idempotent Algebra, Vestnik St. Petersburg University: Mathematics 38 (2005) 45-54.

\bibitem{GenSpeRad1} Y.-Y. Lur, A Max Version of the Generalized Spectral Radius Theorem,
Linear Algebra Appl. 418 (2006) 336-346.

\bibitem{Miettinen} K. Miettinen, Nonlinear multiobjective optimization, Kluwer Academic Publishers, Boston, MA (1999).

\bibitem{Mine} A. Min\'{e}, Weakly relational numerical abstract domains, Ph.D. Thesis, \'Ecole Polytechnique, Palaiseau, France (2004).

\bibitem{NusLAA} R. D. Nussbaum, Convexity and log-convexity for the spectral radius,
Linear Algebra Appl. 73 (1986) 59-122.

\bibitem{GenSpeRad2} A. Peperko, On the Max Version of the Generalized Spectral Radius
Theorem, Linear Algebra Appl. 428 (2008) 2312-2318.

\bibitem{saaty77} T. L. Saaty, A scaling method for priorities in hierarchical structures, J. Math. Psychol. 15 (1977) 234-281.

\bibitem{saaty01} T. L. Saaty, L. G. Vargas, Models, methods, concepts and applications of the analytic hierarchy process, Kluwer Academic Publishers, Norwell, MA (2001).

\bibitem{SergeiVS} S. Sergeev, H. Schneider, P. Butkovi\v{c}, On visualization scaling, subeigenvectors and Kleene stars in max algebra, Linear Algebra Appl. 431 (2009) 2395-2406.

\bibitem{Ngoc} N. M. Tran, Pairwise ranking: choice of method can produce arbitrarily different rank order, arXiv:1103.1110 (2011).

\bibitem{VaidKum} O. S. Vaidya, S. Kumar, Analytic hierarchy process: An overview of applications, Eur. J. Oper. Res. 169 (2006) 1-29.
}
\end{thebibliography}
\end{document}